\documentclass{amsart}
\usepackage{amssymb}
\usepackage{braket}
\usepackage{mathrsfs}
\usepackage[all]{xy}
\usepackage{graphicx}

\newtheorem{thm}{Theorem}[section]
\newtheorem{cor}[thm]{Corollary}
\newtheorem{prop}[thm]{Proposition}
\theoremstyle{definition}
\newtheorem{dfn}[thm]{Definition}

\newtheorem{lem}[thm]{Lemma}
\newtheorem{fact}[thm]{Fact}

\newtheorem{problem}[thm]{Problem}
\newtheorem{nota}[thm]{Notation}
\newtheorem{assumption}[thm]{Assumption}
\theoremstyle{remark}
\newtheorem{rem}[thm]{Remark}
\newtheorem{introdfn}{{\bf Definition} $(${\rm Definition} \ref{DefHeartEquiv}$)$}

\newtheorem{introproblem}{{\bf Problem} $(${\rm Problem} \ref{ProbOne}$)$}

\newtheorem{introthm}{{\bf Theorem} $(${\rm Theorem} \ref{MainThm}$)$}

\newcommand{\C}{\Ccal}               
\newcommand{\K}{\Kcal}               

\newcommand{\Cp}{\C^+}               
\newcommand{\Cm}{\C^-}               
\newcommand{\tp}{\tau^+}             
\newcommand{\tm}{\tau^-}             

\newcommand{\ST}{(\Scal,\Tcal)}      
\newcommand{\UV}{(\Ucal,\Vcal)}      
\newcommand{\STUV}{(\ST,\UV)}        

\newcommand{\CW}{\C/\Wcal}           
\newcommand{\CpW}{\Cp/\Wcal}         
\newcommand{\CmW}{\Cm/\Wcal}         
\newcommand{\HW}{\Hcal/\Wcal}        

\newcommand{\Hbar}{\ovl{H}}          

\newcommand{\Hom}{\mathrm{Hom}}          

\newcommand{\Cpp}{\C^{\prime +}}     
\newcommand{\Cmp}{\C^{\prime -}}     
\newcommand{\tpp}{\tau^{\prime+}}    
\newcommand{\tmp}{\tau^{\prime-}}    

\newcommand{\STp}{(\Scal\ppr,\Tcal\ppr)}      
\newcommand{\UVp}{(\Ucal\ppr,\Vcal\ppr)}      
\newcommand{\STUVp}{(\STp,\UVp)}              

\newcommand{\CWp}{\C\ppr/\Wcal\ppr}   
\newcommand{\HWp}{\Hcal\ppr/\Wcal\ppr}

\newcommand{\Hbarp}{\ovl{H}\ppr}         

\newcommand{\Exto}{\mathrm{Ext}^1}            

\newcommand{\Id}{\mathrm{Id}}         
\newcommand{\se}{\subseteq}           
\newcommand{\quot}{\mathrm{quot.}}    
\newcommand{\exact}{\mathrm{exact}}   
\newcommand{\ppr}{^{\prime}}          
\newcommand{\ppp}{^{\perp}}           
\newcommand{\co}{\colon}              
\newcommand{\ci}{\circ}               
\newcommand{\iv}{^{-1}}               

\newcommand{\lra}{\longrightarrow}    
\newcommand{\tc}{\Rightarrow}         
\newcommand{\hr}{\hookrightarrow}     
\newcommand{\LR}{\Leftrightarrow}     

\newcommand{\ze}{\zeta}         

\newcommand{\Ccal}{\mathcal{C}} 
\newcommand{\Dcal}{\mathcal{D}} 
\newcommand{\Hcal}{\mathcal{H}} 
\newcommand{\Ical}{\mathcal{I}} 
\newcommand{\Kcal}{\mathcal{K}} 
\newcommand{\Mcal}{\mathcal{M}} 
\newcommand{\Pcal}{\mathcal{P}} 
\newcommand{\Scal}{\mathcal{S}} 
\newcommand{\Tcal}{\mathcal{T}} 
\newcommand{\Ucal}{\mathcal{U}} 
\newcommand{\Vcal}{\mathcal{V}} 
\newcommand{\Wcal}{\mathcal{W}} 
\newcommand{\Sbf}{\mathbf{S}}   
\newcommand{\add}{\mathrm{add}}   

\newcommand{\und}{\underline}   
\newcommand{\ovl}{\overline}    
\newcommand{\ov}{\overset}      

\numberwithin{equation}{section}

\begin{document}

\title[Equivalence of hearts of twin cotorsion pairs]{Equivalence of hearts of twin cotorsion pairs on triangulated categories}

\author{Hiroyuki NAKAOKA}
\address{Research and Education Assembly, Science and Engineering Area, Research Field in Science, Kagoshima University, 1-21-35 Korimoto, Kagoshima, 890-0065 Japan\ /\ LAMFA, Universit\'{e} de Picardie-Jules Verne, 33 rue St Leu, 80039 Amiens Cedex1, France}

\email{nakaoka@sci.kagoshima-u.ac.jp}
\urladdr{http://www.lamfa.u-picardie.fr/nakaoka/}

\thanks{The author wishes to thank Professor Yann Palu for his interest and comments.}

\thanks{This work is supported by This work was supported by JSPS KAKENHI Grant Numbers 25800022,\, 24540085.}

\begin{abstract}
In this article, we investigate the condition for the hearts of twin cotorsion pairs to be equivalent, compatibly with the associated functors. This is related to the vanishing of components of pairs through the associated functors. 
\end{abstract}

\maketitle


\section{Introduction}

A cotorsion pair $\UV$, essentially equal to the notion of a torsion pair (\cite{IY}) on a triangulated category $\C$, is a unifying notion of $t$-structure (\cite{BBD}) and cluster tilting subcategory (\cite{KR}, \cite{KZ}).
Generalizing the case of $t$-structures and cluster tilting subcategories, a sequence of additive full subcategories $\Wcal\se\Hcal\se\Ccal$ is associated to any cotorsion pair, which gives an abelian category $\HW$ and a cohomological functor $H\co\C\to\HW$ (\cite{N1}, \cite{AN}). We call $\HW$ the {\it heart} of $\UV$.

In a work by Buan and Marsh (\cite{BM}), a generalization of (\cite{KZ}) has been given. From a rigid object in a triangulated category which is Krull-Schmidt $k$-linear $\Hom$-finite over a field $k$ equipped with a Serre functor, an integral preabelian category is constructed as an ideal quotient there.
This can be regarded as a {\it heart} of a pair of cotorsion pairs as in \cite{N2}.
We call a pair of cotorsion pairs $\Pcal=\STUV$ a {\it twin cotorsion pair} if it satisfies $\Exto(\Scal,\Vcal)=0$ (\cite{N2}). In fact, in the same manner as for the single cotorsion pair, a sequence of additive full subcategories $\Wcal\se\Hcal\se\Ccal$ is associated to any twin cotorsion pair, and it gives a preabelian category $\HW$ called {\it heart} and an additive functor $H\co\C\to\HW$.

Recently, Marsh and Palu (\cite{MP}) have established an equivalence of ideal quotients associated to rigid objects related by a mutation. In this article, to interpret this as an equivalence of hearts, we investigate a condition for the hearts to be equivalent.
Equivalences which we consider are the natural ones, namely, those compatible with associated functors as follows.
\begin{introdfn}
Let $\C$ and $\C\ppr$ be triangulated categories. Let $\Pcal=\STUV$ and $\Pcal\ppr=\STUVp$ be twin cotorsion pairs on $\C$ and $\C\ppr$, respectively. Let $F\co\C\ov{\simeq}{\lra}\C\ppr$ be a triangle equivalence. $\Pcal$ is said to be {\it heart-equivalent to $\Pcal\ppr$ along $F$} if there exists an equivalence of categories $E\co\HW\ov{\simeq}{\lra}\HWp$ which makes the following diagram commutative up to a natural isomorphism.
\[
\xy
(-10,7)*+{\C}="0";
(10,7)*+{\C\ppr}="2";
(-10,-7)*+{\HW}="4";
(10,-7)*+{\HWp}="6";
{\ar^{F}_{\simeq} "0";"2"};
{\ar_{H} "0";"4"};
{\ar^{H\ppr} "2";"6"};
{\ar_{E}^{\simeq} "4";"6"};
{\ar@{}|\circlearrowright "0";"6"};
\endxy
\]
\end{introdfn}
With this terminology, the problem we consider in this article is stated as follows.
\begin{introproblem}
Let $\C,\C\ppr,\Pcal,\Pcal\ppr$ and $F$ be as above.
Under which condition, does the following {\rm (I)} and {\rm (II)} become equivalent?
\begin{enumerate}
\item[{\rm (I)}] $\Pcal$ is heart-equivalent to $\Pcal\ppr$ along $F$.
\item[{\rm (II)}] The following conditions are satisfied.
\begin{itemize}
\item[{\rm (i)}] $H\ppr F(\Ucal)=H\ppr F(\Tcal)=0$.
\item[{\rm (ii)}] $HF\iv(\Ucal\ppr)=HF\iv(\Tcal\ppr)=0$.
\end{itemize}
\end{enumerate}
\end{introproblem}
It can be easily confirmed that {\rm (I)} always implies {\rm (II)}. Thus our goal is to find a condition, with which {\rm (II)} implies {\rm (I)}.
In our main theorem, this is given as follows.
\begin{introthm}
In the above situation, if the conditions 
\begin{itemize}
\item[{\rm (B)}] \ $F(\Scal)\se\Scal\ppr\ast\Vcal\ppr [1]$, \ $F(\Vcal)\se\Scal\ppr [-1]\ast\Vcal\ppr$,
\item[{\rm (C)}] \ $F\iv(\Scal\ppr)\se\Scal\ast\Vcal [1]$, \ $F\iv(\Vcal\ppr)\se\Scal [-1]\ast\Vcal$,
\end{itemize}
are satisfied, then the above {\rm (I)} and {\rm (II)} become equivalent.
\end{introthm}

In section 2, we review the definitions and results used in this article, mainly from \cite{N2}. In section 3, we reduce the problem to a manageable, equivalent one. In section 4 we introduce what happens in the case of single cotorsion pairs. Following the argument by Zhou and Zhu (\cite{ZZ}), conditions {\rm (I)} and {\rm (II)} are shown to be equivalent in this case, without any extra assumption. In section 5, for general twin cotorsion pairs, we give a sufficient condition for {\rm (II)} to imply {\rm (I)}. In section 6, we demonstrate how this condition can be applied to the equivalence given in \cite{MP}.

We also would like to remark that hearts of twin cotorsion pairs on exact categories and their equivalences have been studied by Liu in \cite{L1} and \cite{L2}.

\bigskip

For any category $\K$, we write abbreviately $X\in\K$, to indicate $X$ is an object of $\K$.
For any $X,Y\in\K$, let $\K(X,Y)$ denote the set of morphisms from $X$ to $Y$.
If $\K$ is additive, then for a subcategory $\Mcal\se\K$, we define its right perpendicular category $\Mcal\ppp$ to be the full subcategory of $\K$ consisting of those $X\in\K$ satisfying $\K(\Mcal,X)=0$. Dually, ${}\ppp\Mcal$ denotes the full subcategory of those $X$ satisfying $\K(X,\Mcal)=0$.
If $\Ical\subseteq\K$ is a full additive subcategory, then $\K/\Ical$ is defined to be the ideal quotient of $\K$ by $\Ical$. Namely, $\K/\Ical$ is an additive category defined as follows.
\begin{itemize}
\item[-] Objects in $\K/\Ical$ are the same as those in $\K$.
\item[-] For any $X,Y\in\K$, the morphism set is defined by
\[ (\K/\Ical)(X,Y)=\K(X,Y)/\{ f\in\K(X,Y) \mid f\ \text{factors through some}\ I\in\Ical\}. \]
\end{itemize}

\section{Review of Definitions and results}

Throughout this article, $\C$ denotes a triangulated category with suspension functor $[1]$.
\begin{dfn}
Let $\Ucal,\Vcal\se\C$ be a full subcategories closed under isomorphisms, finite direct sums and summands.
The pair $\UV$ is a {\it cotorsion pair on} $\C$ if it satisfies the following.
\begin{itemize}
\item[{\rm (i)}] $\C=\Ucal\ast\Vcal [1]$.
\item[{\rm (ii)}] $\Exto\UV=0$, where $\Exto(X,Y)=\C(X,Y[1])$ for any $X,Y\in\C$.
\end{itemize}
Here, $\Ucal\ast\Vcal [1]$ denotes the full subcategory of $\C$ consisting of those $C\in\C$ admitting a distinguished triangle
\[ U\rightarrow C\rightarrow V[1]\rightarrow U[1] \]
with $U\in\Ucal$ and $V\in\Vcal$.
\end{dfn}

\begin{dfn}
Let $\ST,\UV$ be two cotorsion pairs on $\C$. The pair $\Pcal=\STUV$ is a {\it twin cotorsion pair on} $\C$ if it satisfies $\Exto(\Scal,\Vcal)=0$. Note that this condition is equivalent to $\Scal\se\Ucal$, and also to $\Vcal\se\Tcal$.
\end{dfn}

Remark that a cotorsion pair $\UV$ on $\C$ can always be regarded as a twin cotorsion pair $(\UV,\UV)$. This is regarded as a {\it degenerated} case of a twin cotorsion pair, as follows.
\begin{dfn}
A twin cotorsion pair $\Pcal=\STUV$ is said to be {\it degenerated to a 
single cotorsion pair} if it satisfies
\[ \Scal=\Ucal\quad \text{and}\quad \Tcal=\Vcal. \]
\end{dfn}

\begin{dfn}\label{DefNotaW}
For a twin cotorsion pair $\Pcal=\STUV$ on $\C$, put
\begin{eqnarray*}
&\Wcal=\Ucal\cap\Tcal,&\\
&\Cp=\Wcal\ast\Vcal [1],\ \ \ \Cm=\Scal [-1]\ast\Wcal,&\\
&\Hcal=\Cp\cap\Cm.&
\end{eqnarray*}
We call the ideal quotient $\Hcal/\Wcal$ the {\it heart} of $\Pcal$. Remark that there are inclusions of full subcategories
\[
\xy
(-16,0)*+{\HW}="0";
(0,8)*+{\CpW}="2";
(0,-8)*+{\CmW}="4";
(16,0)*+{\CW}="6";
{\ar@{^(->} "0";"2"};
{\ar@{^(->} "0";"4"};
{\ar@{^(->} "2";"6"};
{\ar@{^(->} "4";"6"};
{\ar@{}|\circlearrowright "0";"6"};
\endxy.
\]
\end{dfn}

For any morphism $f\co A\to B$ in $\C$, its image under the quotient functor $\C\ov{\quot}{\lra}\CW$ is denoted by $\und{f}\co A\to B$.

\begin{lem}
$\Ucal\se\Cm$ and $\Tcal\se\Cp$ holds.
\end{lem}


\begin{lem}
$\Hcal\cap(\Ucal\ast\Tcal)=\Wcal$ holds.
\end{lem}
\begin{proof}
These are confirmed easily.
\end{proof}

\begin{fact}
$($Corollary 3.8, 3.9 in \cite{N2}.$)$
\begin{enumerate}
\item Inclusion $\CpW\hr\CW$ has a left adjoint $\tp\co\CW\to\CpW$.
\item Inclusion $\CmW\hr\CW$ has a right adjoint $\tm\co\CW\to\CmW$.
\end{enumerate}
\end{fact}

\begin{dfn}
Let $C\in\C$ be any object.
\begin{enumerate}
\item A diagram
\begin{equation}\label{ReflTri}
\xy
(-28,0)*+{S[-1]}="0";
(-8,0)*+{C}="2";
(8,0)*+{Z_C}="4";
(24,0)*+{S}="6";
(-18,-12)*+{U}="8";
(-18,-5)*+{_{\circlearrowright}}="10";
{\ar^{} "0";"2"};
{\ar^{z_C} "2";"4"};
{\ar_{} "4";"6"};
{\ar_{} "0";"8"};
{\ar_{} "8";"2"};
\endxy
\end{equation}
is called a {\it reflection triangle} if it satisfies the following.
\begin{itemize}
\item[{\rm (i)}] $S\in\Scal, U\in\Ucal$.
\item[{\rm (ii)}] $S[-1]\to C\ov{z_C}{\lra}Z_C\to S$ is a distinguished triangle in $\C$.
\item[{\rm (iii)}] $Z_C\in\Cp$.
\end{itemize}


\item Dually, a diagram
\begin{equation}\label{CoreflTri}
\xy
(-24,0)*+{V}="0";
(-8,0)*+{K_C}="2";
(8,0)*+{C}="4";
(28,0)*+{V[1]}="6";
(18,-12)*+{T}="8";
(18,-5)*+{_{\circlearrowright}}="10";
{\ar^{} "0";"2"};
{\ar^{k_C} "2";"4"};
{\ar^{} "4";"6"};
{\ar_{} "4";"8"};
{\ar_{} "8";"6"};
\endxy
\end{equation}
is called a {\it coreflection triangle} if it satisfies the following.
\begin{itemize}
\item[{\rm (i)}] $V\in\Vcal, T\in\Tcal$.
\item[{\rm (ii)}] $V\to K_C\ov{k_C}{\lra}C\to V[1]$ is a distinguished triangle in $\C$.
\item[{\rm (iii)}] $K_C\in\Cm$.
\end{itemize}
\end{enumerate}
\end{dfn}

\begin{fact}$($Corollary 3.8 in \cite{N2}.$)$
For any reflection triangle $(\ref{ReflTri})$, there exists a unique morphism $\ze\in(\CW)(\tp(C),Z_C)$ which is compatible with the adjunction $C\to\tp(C)$ as follows.
\[
\xy
(-10,6)*+{C}="0";
(10,6)*+{Z_C}="2";
(0,-8)*+{\tp(C)}="4";
(0,10)*+{}="5";
{\ar^{\und{z}_C} "0";"2"};
{\ar_(0.4){\mathrm{adjunction}} "0";"4"};
{\ar^(0.4){\ze} "2";"4"};
{\ar@{}|\circlearrowright "4";"5"};
\endxy
\]
This $\ze$ is an isomorphism in $\CW$.

Dually for coreflection triangles.
\end{fact}


\begin{rem}\label{RemExistRef}$($Definition 3.4 in \cite{N2}.$)$
For any $C\in\C$, a reflection triangle always exists. Indeed, by the condition $\C=\Scal\ast\Tcal [1]=\Ucal\ast\Vcal [1]$ and the octahedron axiom, we can draw a diagram
\[
\xy
(-16.2,16)*+{S_C[-1]}="0";
(-14.9,-1)*+{U_C}="2";
(-13.5,-17.5)*+{T_C}="4";
(-4,2)*+{C}="6";
(2,-5.2)*+{Z_C}="8";
(17,8.5)*+{V_C[1]}="10";
(-8.5,-7.5)*+_{_{\circlearrowright}}="12";
(-11.5,5.5)*+_{_{\circlearrowright}}="14";
(4.3,0.5)*+_{_{\circlearrowright}}="14";
{\ar_{} "0";"2"};
{\ar_{} "2";"4"};
{\ar^{} "0";"6"};
{\ar_{} "2";"6"};
{\ar_{z_C} "6";"8"};
{\ar^{} "6";"10"};
{\ar_{} "4";"8"};
{\ar_{} "8";"10"};
\endxy
\]
satisfying $S_C\in\Scal, U_C\in\Ucal, V_C\in\Vcal, T_C\in\Tcal$, in which
\begin{eqnarray*}
&S_C[-1]\to C\to Z_C\to S_C,\quad S_C[-1]\to U_C\to T_C\to S_C,&\\
&U_C\to C\to V_C[1]\to U_C[1],\quad T_C\to Z_C\to V_C[1]\to T_C[1]&
\end{eqnarray*}
are distinguished triangles. Since $T_C\in\Tcal\cap(\Ucal\ast\Scal)=\Wcal$, it follows $Z_C\in\Cp$.

Dually for the existence of coreflection triangles.
\end{rem}


Reflection triangles are \lq\lq functorial" in the following sense.
\begin{rem}\label{RemFunctRefl}
Let $A,B\in\C$ be any object, and let
\begin{equation}\label{ReflTriA}
\xy
(-28,0)*+{S_A[-1]}="0";
(-8,0)*+{A}="2";
(8,0)*+{Z_A}="4";
(24,0)*+{S_A}="6";
(-18,-12)*+{U_A}="8";
(-18,-5)*+{_{\circlearrowright}}="10";
{\ar^{} "0";"2"};
{\ar^{z_A} "2";"4"};
{\ar_{} "4";"6"};
{\ar_{} "0";"8"};
{\ar_{} "8";"2"};
\endxy
\end{equation}
\begin{equation}\label{ReflTriB}
\xy
(-28,0)*+{S_B[-1]}="0";
(-8,0)*+{B}="2";
(8,0)*+{Z_B}="4";
(24,0)*+{S_B}="6";
(-18,-12)*+{U_B}="8";
(-18,-5)*+{_{\circlearrowright}}="10";
{\ar^{} "0";"2"};
{\ar^{z_B} "2";"4"};
{\ar_{} "4";"6"};
{\ar_{} "0";"8"};
{\ar_{} "8";"2"};
\endxy
\end{equation}
be reflection triangles. For any morphism $f\in\C(A,B)$, there exists a morphism of diagrams from $(\ref{ReflTriA})$ to $(\ref{ReflTriB})$ as
\[
\xy
(-26,7)*+{S_A[-1]}="0";
(-7,7)*+{A}="2";
(8,7)*+{Z_A}="4";
(22,7)*+{S_A}="6";
(-16,6)*+{}="7";
(-16,16)*+{U_A}="8";
(-26,-7)*+{S_B[-1]}="10";
(-7,-7)*+{B}="12";
(8,-7)*+{Z_B}="14";
(22,-7)*+{S_B}="16";
(-16,-6)*+{}="17";
(-16,-16)*+{U_B}="18";
{\ar^{} "0";"2"};
{\ar^{z_A} "2";"4"};
{\ar^{} "4";"6"};
{\ar^{} "0";"8"};
{\ar^{} "8";"2"};
{\ar^{} "0";"10"};
{\ar^{f} "2";"12"};
{\ar^{g} "4";"14"};
{\ar^{} "6";"16"};
{\ar^{} "10";"12"};
%
{\ar_{z_B} "12";"14"};
{\ar_{} "14";"16"};
{\ar^{} "10";"18"};
{\ar^{} "18";"12"};
{\ar@{}|\circlearrowright "0";"12"};
{\ar@{}|\circlearrowright "2";"14"};
{\ar@{}|\circlearrowright "4";"16"};
{\ar@{}|\circlearrowright "7";"8"};
{\ar@{}|\circlearrowright "17";"18"};
\endxy
\]
In fact, if we take a distinguished triangle
\[ W\to Z_A\to V[1]\to W[1]\qquad(W\in\Wcal,V\in\Vcal), \]
then in the diagram
\[
\xy
(0,12)*+{S_A[-1]}="0";
(0,4)*+{A}="1";
(0,-4)*+{B}="2";
(0,-12)*+{Z_B}="3";
(-16,-12)*+{W}="4";
(16,-12)*+{V[1]}="5";
(12,0)*+{U_A}="6";
(-2,0)*+{}="7";
{\ar^{} "0";"1"};
{\ar_{f} "1";"2"};
{\ar_{z_B} "2";"3"};
{\ar^{} "4";"3"};
{\ar^{f} "3";"5"};
{\ar^{} "0";"6"};
{\ar_{} "6";"3"};
{\ar@{}|\circlearrowright "6";"7"};
\endxy,
\]
the equality $\C(U_A,V[1])=0$ implies that the composed morphism $S_A[-1]\to Z_B$ factors through $W$, and thus should be the zero morphism, by $\C(S_A[-1],W)=0$.

Also remark that $\und{g}\in(\CpW)(Z_A,Z_B)$ with the commutativity $\und{g}\ci\und{z}_A=\und{z}_B\ci\und{f}$ is unique by the adjoint property.

Dually for coreflection triangles.
\end{rem}

\begin{rem}\label{RemTauUT}$($Lemmas 3.10, 3.11 in \cite{N2}.$)$
For any $C\in\C$, the following holds.
\begin{enumerate}
\item $\tp(C)=0\ \ \LR\ \ C\in\Ucal$.
\item $\tm(C)=0\ \ \LR\ \ C\in\Tcal$.
\end{enumerate}
\end{rem}

\begin{fact}\label{FactUSTV}
$($Lemmas 2.12, 2.13 in \cite{N2}.$)$
\begin{enumerate}
\item If $U[-1]\to A\to B\to U$ is a distinguished triangle satisfying $U\in \Ucal$, then
\[ A\in\Cm\ \ \tc\ \ B\in\Cm. \]
\item If $S[-1]\to A\to B\to S$ is a distinguished triangle satisfying $S\in \Scal$, then
\[ A\in\Cm\ \ \LR\ \ B\in\Cm. \]
\item If $T\to A\to B\to T[1]$ is a distinguished triangle satisfying $T\in \Tcal$, then
\[ B\in\Cp\ \ \tc\ \ A\in\Cp. \]
\item If $V\to A\to B\to V[1]$ is a distinguished triangle satisfying $V\in \Vcal$, then
\[ B\in\Cp\ \ \LR\ \ A\in\Cp. \]
\end{enumerate}
\end{fact}

\begin{prop}
For the compositions of functors
\begin{eqnarray*}
&\tp\tm:=\big(\CW\ov{\tm}{\lra}\CmW\hr\CW\ov{\tp}{\lra}\CpW\hr\CW \big),&\\
&\tm\tp:=\big(\CW\ov{\tp}{\lra}\CpW\hr\CW\ov{\tm}{\lra}\CmW\hr\CW \big),&
\end{eqnarray*}
there exists a natural isomorphism $\tp\tm\cong\tm\tp$.
\end{prop}
\begin{proof}
For any $C\in\C$, take a reflection triangle $(\ref{ReflTri})$ and a coreflection triangle $(\ref{CoreflTri})$. 
The equalities $\C(S_C[-1],T_C)=0$ and $\C(U_C,V_C[1])=0$ give morphisms $Z_C\to T_C$ and $U_C\to K_C$, which make
\[
\xy
(16,-13)*+{V_C[1]}="0";
(10,8)*+{Z_C}="2";
(0,-1)*+{C}="6";
(-10,8)*+{K_C}="8";
(-16,-13)*+{S_C[-1]}="10";
(-24,-3)*+{U_C}="12";
(24,-3)*+{T_C}="14";
(10.5,1.5)*+_{_{\circlearrowright}}="20";
(-10.5,1.5)*+_{_{\circlearrowright}}="22";
(-14,-7)*+_{_{\circlearrowright}}="24";
(14,-7)*+_{_{\circlearrowright}}="26";
{\ar^{} "6";"14"};
{\ar_{} "6";"0"};
{\ar^{z_C} "6";"2"};
{\ar^{k_C} "8";"6"};
{\ar_{} "10";"6"};
{\ar^{} "12";"6"};
{\ar_{} "10";"12"};
{\ar^{} "12";"8"};
{\ar^{} "2";"14"};
{\ar_{} "14";"0"};
\endxy
\]
commutative. Then the composed morphisms $S_C[-1]\to K_C$ and $Z_C\to V_C[1]$ make the following diagram commutative.
\[
\xy
(19,-13)*+{V_C[1]}="0";
(8,4)*+{Z_C}="2";
(0,-2)*+{C}="6";
(-8,4)*+{K_C}="8";
(-19,-13)*+{S_C[-1]}="10";
(-26,-4)*+{U_C}="12";
(26,-4)*+{T_C}="14";
(9,-3)*+_{_{\circlearrowright}}="20";
(-9,-3)*+_{_{\circlearrowright}}="22";
(-19,-5)*+_{_{\circlearrowright}}="24";
(19,-5)*+_{_{\circlearrowright}}="26";
{\ar^{} "2";"0"};
{\ar_{} "6";"0"};
{\ar_{z_C} "6";"2"};
{\ar_{k_C} "8";"6"};
{\ar_{} "10";"6"};
{\ar^{} "10";"8"};
{\ar_{} "10";"12"};
{\ar^{} "12";"8"};
{\ar^{} "2";"14"};
{\ar_{} "14";"0"};
\endxy
\]

If we complete $S_C[-1]\to K_C$ into a distinguished triangle $S_C[-1]\to K_C\to Q\to S_C$, then by the octahedron axiom, we obtain a diagram
\[
\xy
(19,-13)*+{V_C[1]}="0";
(8,4)*+{Z_C}="2";
(0,16)*+{Q}="4";
(0,-2)*+{C}="6";
(-8,4)*+{K_C}="8";
(-19,-13)*+{S_C[-1]}="10";
(0,7)*+_{_{\circlearrowright}}="12";
(9,-3)*+_{_{\circlearrowright}}="14";
(-9,-3)*+_{_{\circlearrowright}}="16";
{\ar^{} "2";"0"};
{\ar^{} "4";"2"};
{\ar_{} "6";"0"};
{\ar_{z_C} "6";"2"};
{\ar_{k_C} "8";"6"};
{\ar_{} "10";"6"};
{\ar^{} "8";"4"};
{\ar^{} "10";"8"};
\endxy
\]
 in which
\begin{eqnarray*}
&S_C[-1]\to K_C\to Q\to S_C,\quad S_C[-1]\to C\to Z_C\to S_C,&\\
&K_C\to C\to V_C[1]\to K_C[1],\quad Q\to Z_C\to V_C[1]\to Q[1]&
\end{eqnarray*}
are distinguished triangles. By Fact \ref{FactUSTV},
\begin{itemize}
\item[-] $Z_C\in\Cp$ implies $Q\in\Cp$,
\item[-] $K_C\in\Cm$ implies $Q\in\Cm$.
\end{itemize}
Thus we have $Q\in\Hcal$, and
\[
\xy
(-28,0)*+{S_C[-1]}="0";
(-8,0)*+{K_C}="2";
(8,0)*+{Q}="4";
(24,0)*+{S_C}="6";
(-18,-12)*+{U_C}="8";
(-18,-5)*+{_{\circlearrowright}}="10";
{\ar^{} "0";"2"};
{\ar^{} "2";"4"};
{\ar_{} "4";"6"};
{\ar_{} "0";"8"};
{\ar_{} "8";"2"};
\endxy
\]
\[
\xy
(-24,0)*+{V_C}="0";
(-8,0)*+{Q}="2";
(8,0)*+{Z_C}="4";
(28,0)*+{V_C[1]}="6";
(18,-12)*+{T_C}="8";
(18,-5)*+{_{\circlearrowright}}="10";
{\ar^{} "0";"2"};
{\ar^{} "2";"4"};
{\ar^{} "4";"6"};
{\ar_{} "4";"8"};
{\ar_{} "8";"6"};
\endxy
\]
are reflection and coreflection triangles, respectively. This implies $\tp(K_C)\cong Q\cong\tm(Z_C)$.
\end{proof}


\begin{cor}
$\tp(\CmW)\se\HW$ and $\tm(\CpW)\se\HW$ hold.
\end{cor}

\begin{dfn}
Define functors $h$ and $H$ to be the compositions
\begin{eqnarray*}
h&=&\tp\tm\ =\ \big(\CW\ov{\tm}{\lra}\CmW\ov{\tp}{\lra}\HW\big),\\
H&=&\big(\C\ov{\quot}{\lra}\CW\ov{h}{\lra}\HW\big).
\end{eqnarray*}
Also, let $\iota$ be the inclusion functor
\[ \iota\co\HW\hr\CW. \]
We have a natural isomorphism $h\ci\iota\cong\Id_{\HW}$.
\end{dfn}

\begin{prop}\label{RemHTauU}
For any $C\in\C$, the following are equivalent.
\begin{enumerate}
\item $H(C)=0$.
\item $\tm(C)\in\Ucal/\Wcal$.
\item $\tp(C)\in\Tcal/\Wcal$.
\end{enumerate}
\end{prop}
\begin{proof}
This follows from Remark \ref{RemTauUT}.
\end{proof}


\begin{prop}
$H(\Ucal\ast\Tcal)=0$ holds. Thus $H$ factors through $\C\ov{\quot}{\lra}\C/\Ucal\ast\Tcal$ as follows.
\[
\xy
(-6,6)*+{\C}="0";
(0,-10)*+{}="1";
(-10,-6)*+{\C/\Ucal\ast\Tcal}="2";
(10,-6)*+{\HW}="4";
{\ar_{\quot} "0";"2"};
{\ar^{H} "0";"4"};
{\ar^{} "2";"4"};
{\ar@{}|\circlearrowright "0";"1"};
\endxy
\]
In particular we have $H(\Ucal)=H(\Tcal)=0$.
\end{prop}
\begin{proof}
Let $U\to C\to T\to U[1]$ be any distinguished triangle satisfying $U\in\Ucal$ and $T\in\Tcal$. If we decompose $T$ into a distinguished triangle
\[ V_T\to U_T\to T\to V_T[1]\quad (U_T\in\Ucal, V_T\in\Vcal), \]
then by the octahedron axiom, 
we obtain a diagram
\[
\xy
(-20,16)*+{U}="0";
(0.5,15)*+{L}="2";
(17,14)*+{U_T}="4";
(-2,4)*+{C}="6";
(6.1,-0.6)*+{T}="8";
(-5.8,-14.2)*+{V_T[1]}="10";
(7.5,8.5)*+_{_{\circlearrowright}}="12";
(-5.5,11.5)*+_{_{\circlearrowright}}="14";
(-0.5,-4.3)*+_{_{\circlearrowright}}="14";
{\ar^{} "0";"2"};
{\ar^{} "2";"4"};
{\ar_{} "0";"6"};
{\ar^{} "2";"6"};
{\ar^{} "6";"8"};
{\ar_{} "6";"10"};
{\ar^{} "4";"8"};
{\ar^{} "8";"10"};
\endxy
\]
 in which
\begin{eqnarray*}
&U\to L\to U_T\to U[1],\quad U\to C\to T\to U[1],&\\
&L\to C\to V_T[1]\to L[1],\quad U_T\to T\to V_T[1]\to U_T[1]&
\end{eqnarray*}
are distinguished triangles.
Since $U,U_T\in\Ucal$ implies $L\in\Ucal\se\Cm$, this gives a coreflection triangle
\[
\xy
(-24,0)*+{V_T}="0";
(-8,0)*+{L}="2";
(8,0)*+{C}="4";
(28,0)*+{V_T[1]}="6";
(18,-12)*+{T}="8";
(18,-5)*+{_{\circlearrowright}}="10";
{\ar^{} "0";"2"};
{\ar^{} "2";"4"};
{\ar^{} "4";"6"};
{\ar_{} "4";"8"};
{\ar_{} "8";"6"};
\endxy
\]
which yields $\tm(C)\cong L\in\Ucal/\Wcal$, and thus $H(C)=0$.
\end{proof}



\section{Problem setting}


Let $\C\ppr$ be another triangulated category,
and let $\Pcal\ppr=(\STp,(\Ucal\ppr, \Tcal\ppr))$ be a twin cotorsion pair
on $\C\ppr$.
\begin{nota}\label{NotaPrime}
We denote the associated categories as
\begin{eqnarray*}
&\Wcal\ppr=\Ucal\ppr\cap\Tcal\ppr,&\\
&\Cpp=\Wcal\ppr\ast\Vcal\ppr[1],\quad \Cmp=\Scal\ppr[-1]\ast\Wcal\ppr,&\\
&\Hcal\ppr=\Cpp\cap\Cmp,&
\end{eqnarray*}
and functors as
\begin{eqnarray*}
&h\ppr=\tpp\tmp\co\CWp\to\HWp,&\\
&H\ppr=\big(\C\ppr\ov{\quot}{\lra}\CWp\ov{h\ppr}{\lra}\HWp\big),&\\
&\iota\ppr\co\HWp\hr\CWp.&
\end{eqnarray*}
\end{nota}


\begin{dfn}\label{DefHeartEquiv}
Let $\C,\C\ppr$ and $\Pcal,\Pcal\ppr$ be as above, and let $F\co\C\ov{\simeq}{\lra}\C\ppr$ be a triangle equivalence. $\Pcal$ is said to be {\it heart-equivalent to $\Pcal\ppr$ along $F$} if there exists an equivalence of categories $E\co\HW\ov{\simeq}{\lra}\HWp$ which makes the following diagram commutative up to a natural isomorphism.
\begin{equation}\label{DiagFE}
\xy
(-10,7)*+{\C}="0";
(10,7)*+{\C\ppr}="2";
(-10,-7)*+{\HW}="4";
(10,-7)*+{\HWp}="6";
{\ar^{F}_{\simeq} "0";"2"};
{\ar_{H} "0";"4"};
{\ar^{H\ppr} "2";"6"};
{\ar_{E}^{\simeq} "4";"6"};
{\ar@{}|\circlearrowright "0";"6"};
\endxy
\end{equation}
Remark that this notion of a heart-equivalence is an equivalence relation, in an obvious sense.
\end{dfn}

\begin{prop}
Let $\C,\C\ppr,\Pcal,\Pcal\ppr$ and $F$ be as above.
Then the following {\rm (I)} implies {\rm (II)}.
\begin{enumerate}
\item[{\rm (I)}] $\Pcal$ is heart-equivalent to $\Pcal\ppr$ along $F$.
\item[{\rm (II)}] The following conditions are satisfied.
\begin{itemize}
\item[{\rm (i)}] $H\ppr F(\Ucal)=H\ppr F(\Tcal)=0$.
\item[{\rm (ii)}] $HF\iv(\Ucal\ppr)=HF\iv(\Tcal\ppr)=0$.
\end{itemize}
\end{enumerate}
\end{prop}
\begin{proof}
This immediately follows from $H(\Ucal)=H(\Tcal)=0$ and $H\ppr(\Ucal\ppr)=H\ppr(\Tcal\ppr)=0$.
\end{proof}

\smallskip

The following is our problem in this article.
\begin{problem}\label{ProbOne}
Conversely, does {\rm (I)} follow from {\rm (II)} with some extra conditions?
\end{problem}

\smallskip

We are going to reduce this problem to a more manageable one (Problem \ref{ProbTwo}).
First, remark that this is reduced to the case $F=\Id$.
\begin{rem}\label{ClaimOne}
Let $F\co\C\ov{\simeq}{\lra}\C\ppr$ be a triangle equivalence, and let $\Pcal=\STUV$ be a  twin cotorsion pair on $\C$.
If we put $F(\Pcal)=((F(\Scal),F(\Tcal)),(F(\Ucal),F(\Vcal)))$, then the following holds.
\begin{enumerate}
\item $F(\Pcal)$ is a twin cotorsion pair on $\C\ppr$.
\item $\Pcal$ is heart-equivalent to $F(\Pcal)$ along $F$.
\end{enumerate}

Thus, replacing $\Pcal$ by $F(\Pcal)$, we may assume $\C=\C\ppr$ and $F=\Id_{\C}$ from the first.
\end{rem}

\medskip

From this we may assume $F=\Id_{\C}$, and $\Pcal\ppr$ is a twin cotorsion pair on $\C$. We will keep using 
Notation \ref{NotaPrime} also in this case. For example, $H\ppr$ denotes a functor $H\ppr\co \C\to\HWp$.

\medskip

Second, note that the candidate for $E$ in $(\ref{DiagFE})$ is unique up to natural transformations.
\begin{rem}\label{ClaimTwo}
Let $\Pcal=\STUV$ and $\Pcal\ppr=\STUVp$ be twin cotorsion pairs on $\C$. Assume that the condition
\[ H\ppr(\Wcal)=0 \]
is satisfied. Then the following holds by the commutativity of the following diagrams.

\[
\xy
(-16,13)*+{\Hcal}="-2";
(-4,-3)*+{}="-1";
(0,8)*+{\C}="0";
(0,-13)*+{}="1";
(-16,-8)*+{\HW}="2";
(16,-8)*+{\HWp}="4";
{\ar@{^(->}^{} "-2";"0"};
{\ar_{\quot} "-2";"2"};
{\ar^{H} "0";"2"};
{\ar^{H\ppr} "0";"4"};
{\ar_{E} "2";"4"};
{\ar@{}|\circlearrowright "0";"1"};
{\ar@{}|\circlearrowright "-2";"-1"};
\endxy
\ ,\quad
\xy
(-16,13)*+{\Hcal}="-2";
(-1,-8)*+{}="-1";
(0,8)*+{\C}="0";
(9,-13)*+{}="1";
(-16,-8)*+{\HW}="2";
(0,-8)*+{\CW}="3";
(16,-8)*+{\HWp}="4";
{\ar@{^(->}^{} "-2";"0"};
{\ar_{\quot} "-2";"2"};
{\ar_{\quot} "0";"3"};
{\ar^{H\ppr} "0";"4"};
{\ar_{\iota} "2";"3"};
{\ar_(0.46){\Hbarp} "3";"4"};
{\ar@{}|\circlearrowright "0";"1"};
{\ar@{}|\circlearrowright "-2";"-1"};
\endxy
\].
\begin{enumerate}
\item $H\ppr$ induces a functor $\Hbarp\co\CW\to\HWp$ which makes the following diagram commutative up to a natural isomorphism.
\[
\xy
(-6,6)*+{\C}="0";
(0,-10)*+{}="1";
(-10,-6)*+{\CW}="2";
(10,-6)*+{\HWp}="4";
{\ar_{\quot} "0";"2"};
{\ar^{H\ppr} "0";"4"};
{\ar_(0.46){\Hbarp} "2";"4"};
{\ar@{}|\circlearrowright "0";"1"};
\endxy
\]
\item If a functor $E\co\HW\to\HWp$ makes the diagram
\[
\xy
(0,7)*+{\C}="0";
(0,-10)*+{}="1";
(-11,-6)*+{\HW}="2";
(11,-6)*+{\HWp}="4";
{\ar_{H} "0";"2"};
{\ar^{H\ppr} "0";"4"};
{\ar_{E} "2";"4"};
{\ar@{}|\circlearrowright "0";"1"};
\endxy
\]
commutative up to a natural isomorphism, then there exists a natural isomorphism
\[ E\cong\Hbarp\ci\iota, \]
where $\iota\co\HW\hr\CW$ is the inclusion.
\end{enumerate}
\end{rem}


Thus, under the conditions $H\ppr(\Wcal)=0$ and $H(\Wcal\ppr)=0$ (remark that condition {\rm (II)} in Problem \ref{ProbOne} implies these conditions), we may restrict our attention to the functors
\[ E=\Hbarp\ci\iota\co\HW\to\HWp \quad\text{and}\quad E\ppr=\Hbar\ci\iota\ppr\co\HWp\to\HW, \]
and it is enough to find a condition which induces
\begin{itemize}
\item[{\rm (a)}] $E\ci H\cong H\ppr$, $E\ppr\ci H\ppr\cong H$,
\item[{\rm (b)}] $E\ci E\ppr\cong\Id$, $E\ppr\ci E\cong\Id$.
\end{itemize}
However, {\rm (b)} follows from {\rm (a)}, as follows.
\begin{lem}\label{ClaimThree}
Let $\Pcal$ and $\Pcal\ppr$ be twin cotorsion pairs on $\C$, 
and assume that $H\ppr(\Wcal)=0$ and $H(\Wcal\ppr)=0$ are satisfied. Let $\Hbarp\co\CW\to\HWp$ and $\Hbar\co\C/\Wcal\ppr\to\HW$ be the unique functors induced from $H\ppr$ and $H$ respectively, as in Remark \ref{ClaimTwo}.
Put
\begin{eqnarray*}
&E=\Hbarp\ci\iota\co\HW\to\HWp,&\\
&E\ppr=\Hbar\ci\iota\ppr\co\HWp\to\HW.&
\end{eqnarray*}
If $E$ and $E\ppr$ satisfy {\rm (a)}, then they also satisfy {\rm (b)}.
\end{lem}
\begin{proof}
$E\ppr\ci E\cong\Id$ follows from the commutativity of the following diagram up to natural isomorphisms.
\[
\xy
(-28,26)*+{\Hcal}="0";
(-28,-24)*+{\HW}="2";
(-10,10)*+{\C}="4";
(-10,-14)*+{\C/\Wcal}="6";
(10,-24)*+{\HWp}="10";
(30,10)*+{\HW}="12";
(-19,0)*+{_{\circlearrowright}}="21";
(-7.5,15)*+{_{\circlearrowright}}="22";
(-4,-10)*+{_{\circlearrowright}}="23";
(-9,-20)*+{_{\circlearrowright}}="25";
(10,-2.5)*+{_{\circlearrowright}}="28";
{\ar_{\quot} "0";"2"};
{\ar@{_(->} "0";"4"};
{\ar^{\quot} "0";"12"};
{\ar^{\iota} "2";"6"};
{\ar_{E} "2";"10"};
{\ar_{\quot} "4";"6"};
{\ar^{H\ppr} "4";"10"};
{\ar^{H} "4";"12"};
{\ar_{\Hbarp} "6";"10"};
{\ar_{E\ppr} "10";"12"};
\endxy
\]
Similarly for $E\ci E\ppr\cong\Id$.
\end{proof}


By Remarks \ref{ClaimOne}, \ref{ClaimTwo} and Lemma \ref{ClaimThree}, Problem \ref{ProbOne} has been reduced to the following.
\begin{problem}\label{ProbTwo}
Let $\Pcal=\STUV$ and $\Pcal\ppr=\STUVp$ be twin cotorsion pairs on $\C$. Suppose that the condition
\[ H\ppr(\Ucal)=H\ppr(\Tcal)=0 \]
is satisfied. Let $\Hbarp\co\CW\to\HWp$ be the functor induced from $H\ppr$, and put $E=\Hbarp\ci\iota$. With some extra condition, does there exist a natural isomorphism $E\ci H\cong H\ppr$?
\end{problem}


\section{Degenerated case}
Consider the case each of $\Pcal$ and $\Pcal\ppr$ is degenerated to a single cotorsion pair. In this case, it requires no extra condition. This is essentially due to \cite{ZZ} (see also \cite{L2}).

\begin{rem}
If $\Pcal$ is degenerated, then its heart $\HW$ becomes an abelian category (\cite{N1}), and the functor $H\co\C\to\HW$ becomes cohomological (\cite{AN}).
\end{rem}

\begin{prop}\label{PropSingle}
Let $\UV$ and $\UVp$ be cotorsion pairs on $\C$. Suppose that the condition
\[ H\ppr(\Ucal)=H\ppr(\Vcal)=0 \]
is satisfied. Let $\Hbarp\co\CW\to\HWp$ be the functor induced from $H$, and put $E=\Hbarp\ci\iota$. Then, $E\ci H\cong H\ppr$ holds.
\end{prop}
\begin{proof}
For each object $A\in\C$, choose a reflection triangle $(\ref{ReflTriA})$, and a coreflection triangle
\[
\xy
(-24,0)*+{V_A}="0";
(-8,0)*+{K_A}="2";
(8,0)*+{A}="4";
(28,0)*+{V_A[1]}="6";
(18,-12)*+{T_A}="8";
(18,-5)*+{_{\circlearrowright}}="10";
{\ar^{} "0";"2"};
{\ar^{k_A} "2";"4"};
{\ar^{} "4";"6"};
{\ar_{} "4";"8"};
{\ar_{} "8";"6"};
\endxy.
\]
By Remark \ref{RemFunctRefl}, for any morphism $f\in\C(A,B)$, we obtain a morphism of diagrams
\[
\xy
(26,7)*+{V_A[1]}="0";
(7,7)*+{A}="2";
(-8,7)*+{K_A}="4";
(-22,7)*+{V_A}="6";
(16,6)*+{}="7";
(16,16)*+{T_A}="8";
(26,-7)*+{V_B[1]}="10";
(7,-7)*+{B}="12";
(-8,-7)*+{K_B}="14";
(-22,-7)*+{V_B}="16";
(16,-6)*+{}="17";
(16,-16)*+{T_B}="18";
{\ar^{} "2";"0"};
{\ar^{k_A} "4";"2"};
{\ar^{} "6";"4"};
{\ar^{} "8";"0"};
{\ar^{} "2";"8"};
{\ar^{} "0";"10"};
{\ar_{f} "2";"12"};
{\ar_{g} "4";"14"};
{\ar^{} "6";"16"};
{\ar_{}"12";"10"};
{\ar_{k_B} "14";"12"};
{\ar_{} "16";"14"};
{\ar_{} "18";"10"};
{\ar_{} "12";"18"};
{\ar@{}|\circlearrowright "2";"14"};
{\ar@{}|\circlearrowright "4";"16"};
{\ar@{}|\circlearrowright "7";"8"};
{\ar@{}|\circlearrowright "17";"18"};
{\ar@{}|\circlearrowright "0";"12"};
\endxy
\]
and
\[
\xy
(-28,7)*+{S_{K_A}[-1]}="0";
(-7,7)*+{K_A}="2";
(8,7)*+{Z_{K_A}}="4";
(22,7)*+{S_{K_A}}="6";
(-17,6)*+{}="7";
(-17,16)*+{U_{K_A}}="8";
(-28,-7)*+{S_{K_B}[-1]}="10";
(-7,-7)*+{K_B}="12";
(8,-7)*+{Z_{K_B}}="14";
(22,-7)*+{S_{K_B}}="16";
(-17,-6)*+{}="17";
(-17,-16)*+{U_{K_B}}="18";
{\ar^{} "0";"2"};
{\ar^{z_{K_A}} "2";"4"};
{\ar^{} "4";"6"};
{\ar^{} "0";"8"};
{\ar^{} "8";"2"};
{\ar^{} "0";"10"};
{\ar^{g} "2";"12"};
{\ar^{h} "4";"14"};
{\ar^{} "6";"16"};
{\ar_{} "10";"12"};
{\ar_{z_{K_B}} "12";"14"};
{\ar_{} "14";"16"};
{\ar^{} "10";"18"};
{\ar^{} "18";"12"};
{\ar@{}|\circlearrowright "2";"14"};
{\ar@{}|\circlearrowright "4";"16"};
{\ar@{}|\circlearrowright "7";"8"};
{\ar@{}|\circlearrowright "17";"18"};
{\ar@{}|\circlearrowright "0";"12"};
\endxy,
\]
which imply $H(f)=\tp\tm(\und{f})=\tp(\und{g})=\und{h}$.

Since $H\ppr$ is cohomological and $H\ppr(\Ucal)=H\ppr(\Tcal)=0$ by assumption, we obtain morphisms of exact sequences
\[
\xy
(-27,7)*+{0}="0";
(-11,7)*+{H\ppr(K_A)}="2";
(11,7)*+{H\ppr(A)}="4";
(31,7)*+{H\ppr(V_A[1])}="6";
(50,7)*+{\exact};
(-27,-7)*+{0}="10";
(-11,-7)*+{H\ppr(K_B)}="12";
(11,-7)*+{H\ppr(B)}="14";
(31,-7)*+{H\ppr(V_B[1])}="16";
(50,-7)*+{\exact};
{\ar^{} "0";"2"};
{\ar^(0.54){H\ppr(k_A)} "2";"4"};
{\ar^(0.44){0} "4";"6"};
%
{\ar_{H\ppr(g)} "2";"12"};
{\ar_{H\ppr(f)} "4";"14"};
{\ar^{} "6";"16"};
{\ar_{} "10";"12"};
{\ar_(0.54){H\ppr(k_B)} "12";"14"};
{\ar_(0.44){0} "14";"16"};
%
{\ar@{}|\circlearrowright "2";"14"};
{\ar@{}|\circlearrowright "4";"16"};
\endxy
\]
and
\[
\xy
(-35,7)*+{H\ppr(S_{K_A}[-1])}="0";
(-11,7)*+{H\ppr(K_A)}="2";
(14,7)*+{H\ppr(Z_{K_A})}="4";
(31,7)*+{0}="6";
(42,7)*+{\exact};
(-35,-7)*+{H\ppr(S_{K_B}[-1])}="10";
(-11,-7)*+{H\ppr(K_B)}="12";
(14,-7)*+{H\ppr(Z_{K_B})}="14";
(31,-7)*+{0}="16";
(42,-7)*+{\exact};
{\ar^(0.56){0} "0";"2"};
{\ar^{H\ppr(z_{K_A})} "2";"4"};
{\ar^{} "4";"6"};
{\ar_{} "0";"10"};
{\ar^{H\ppr(g)} "2";"12"};
{\ar^{H\ppr(h)} "4";"14"};
%
{\ar_(0.56){0} "10";"12"};
{\ar_{H\ppr(z_{K_B})} "12";"14"};
{\ar_{} "14";"16"};
{\ar@{}|\circlearrowright "0";"12"};
{\ar@{}|\circlearrowright "2";"14"};
\endxy.
\]
These yield a commutative diagram
\[
\xy
(-37,7)*+{H\ppr(A)}="2";
(-11,7)*+{H\ppr(K_A)}="4";
(15,7)*+{H\ppr(Z_{K_A})}="6";
(30.5,7)*+{=EH(A)};
(-37,-7)*+{H\ppr(B)}="12";
(-11,-7)*+{H\ppr(K_B)}="14";
(15,-7)*+{H\ppr(Z_{K_B})}="16";
(30.5,-7)*+{=EH(B)};
{\ar_{H\ppr(k_A)}^{\cong} "4";"2"};
{\ar^{H\ppr(z_{K_A})}_{\cong} "4";"6"};
{\ar_{H\ppr(f)} "2";"12"};
{\ar^{H\ppr(g)} "4";"14"};
{\ar^{H\ppr(h)=EH(f)} "6";"16"};
{\ar^{H\ppr(k_B)}_{\cong} "14";"12"};
{\ar_{H\ppr(z_{K_B})}^{\cong} "14";"16"};
{\ar@{}|\circlearrowright "2";"14"};
{\ar@{}|\circlearrowright "4";"16"};
\endxy.
\]
Thus if we put
\[ \eta_A=H\ppr(z_{K_A})\ci H\ppr(k_A)\iv, \]
then $\eta=\{\eta_A\}_{A\in\C}$ gives a natural isomorphism $\eta\co H\ppr\ov{\cong}{\Longrightarrow}E\ci H$.
\end{proof}

In terms of the original problem (Problem \ref{ProbOne}), we obtain the following.
\begin{cor}
Let $\Pcal=\UV$ and $\Pcal\ppr=\UVp$ be cotorsion pairs on $\C$ and $\C\ppr$, respectively.
Let $F\co\C\ov{\simeq}{\lra}\C\ppr$ be a triangle equivalence. Then, the following are equivalent.
\begin{enumerate}
\item $\Pcal$ is heart-equivalent to $\Pcal\ppr$ along $F$.
\item $H\ppr F(\Ucal)=H\ppr F(\Vcal)=0$ and $HF\iv(\Ucal\ppr)=HF\iv(\Vcal\ppr)=0$ are satisfied.
\end{enumerate}
\end{cor}
\begin{proof}
This immediately follows from Proposition \ref{PropSingle}.
\end{proof}


\section{A sufficient condition in general case}

We return to the general case of twin cotorsion pairs. Let $\Pcal=\STUV$ be a twin cotorsion pair on $\C$. We start with some improvements of the results from \cite{N2}.

The following gives a partial converse to Lemma 5.1 in \cite{N2}.
\begin{prop}\label{PropCokOne}
Let $S[-1]\ov{e}{\lra}A\ov{f}{\lra}B\ov{g}{\lra}S$ be a distinguished triangle with $S\in\Scal$. Then,
\[ H(S[-1])\ov{H(e)}{\lra}H(A)\ov{H(f)}{\lra}H(B)\to 0 \]
is a cokernel sequence in $\HW$.
\end{prop}
\begin{proof}
By the adjoint property of $\tp$, it is enough to show that the sequence
\begin{equation}\label{SeqE}
0\to(\CW)(\tm(B),Y)\ov{-\ci\tm(\und{f})}{\lra}(\CW)(\tm(A),Y)\ov{-\ci\tm(\und{e})}{\lra}(\CW)(\tm(S[-1]),Y)
\end{equation}
is exact for any $Y\in\Cp$.

Take a coreflection triangle
\[
\xy
(-24,0)*+{V_B}="0";
(-8,0)*+{K_B}="2";
(8,0)*+{B}="4";
(28,0)*+{V_B[1]}="6";
(18,-12)*+{T_B}="8";
(18,-5)*+{_{\circlearrowright}}="10";
{\ar^{} "0";"2"};
{\ar^{k_B} "2";"4"};
{\ar^{v_B} "4";"6"};
{\ar_{} "4";"8"};
{\ar_{} "8";"6"};
\endxy
\]
and complete $v_B\ci f\co B\to V_B[1]$ into a distinguished triangle
\[ V_B\to L\ov{\ell}{\lra}A\ov{v_B\ci f}{\lra}V_B[1]. \]
By the octahedron axiom, we obtain a diagram
\[
\xy
(-20,16)*+{S[-1]}="0";
(0.5,15)*+{L}="2";
(17,14)*+{K_B}="4";
(-2,4)*+{A}="6";
(6.1,-0.6)*+{B}="8";
(-5.8,-14.2)*+{V_B[1]}="10";
(7.5,8.5)*+_{_{\circlearrowright}}="12";
(-5.5,11.5)*+_{_{\circlearrowright}}="14";
(-0.5,-4.3)*+_{_{\circlearrowright}}="14";
{\ar^{s} "0";"2"};
{\ar^{h} "2";"4"};
{\ar_{e} "0";"6"};
{\ar^{\ell} "2";"6"};
{\ar^{f} "6";"8"};
{\ar_{} "6";"10"};
{\ar^{k_B} "4";"8"};
{\ar^{v_B} "8";"10"};
\endxy
\]
in which
\begin{eqnarray*}
&S[-1]\to L\to K_B\to S,\quad S[-1]\to A\to B\to S,&\\
&L\to A\to V_B[1]\to L[1],\quad K_B\to B\to V_B[1]\to K_B[1]&
\end{eqnarray*}
are distinguished triangles.
Then by Fact \ref{FactUSTV}, $K_B\in\Cm$ implies $L\in\Cm$, and thus
\[
\xy
(-24,0)*+{V_B}="0";
(-8,0)*+{L}="2";
(8,0)*+{A}="4";
(28,0)*+{V_B[1]}="6";
(18,-12)*+{T_B}="8";
(18,-5)*+{_{\circlearrowright}}="10";
{\ar^{} "0";"2"};
{\ar^{\ell} "2";"4"};
{\ar^{v_B\ci f} "4";"6"};
{\ar_{} "4";"8"};
{\ar_{} "8";"6"};
\endxy
\]
becomes a coreflection triangle. Thus we may assume $\tm(\und{f})=\und{h}$. Since $S[-1]\in\Scal [-1]\se\Cm$, we may also assume $\tm(\und{e})=\und{s}$.

\medskip

 {\rm (1)} Let $X\in\C$ be any object, and let $x\in\C(L,X)$ be any morphism.
If $x$ satisfies $\und{x}\ci\und{s}=0$, then $x\ci s$ factors through some $W\in\Wcal$ as follows.
\[
\xy
(-8,0)*+{S[-1]}="2";
(8,0)*+{L}="4";
(24,0)*+{K_B}="6";
(-8,-12)*+{W}="8";
(8,-12)*+{X}="10";
{\ar^(0.54){s} "2";"4"};
{\ar^(0.46){h} "4";"6"};
{\ar_{} "2";"8"};
{\ar^{x} "4";"10"};
{\ar_{} "8";"10"};
{\ar@{}|\circlearrowright "2";"10"};
\endxy
\]
Since $\C(S[-1],W)=0$ it follows $x\ci s=0$, and thus $x$ factors through $h$.

\smallskip
 {\rm (2)} Let $Y\in\Cp$ be any object, and let $y\in\C(K_B,Y)$ be any morphism. Decompose $Y$ into a distinguished triangle
\[ V_Y\to W_Y\ov{w_Y}{\lra}Y\ov{v_Y}{\lra}V_Y[1] \]
satisfying $V_Y\in\Vcal, W_Y\in\Wcal$.
If $y$ satisfies $\und{y}\ci\und{h}=0$, then $y\ci h$ factors through $w_Y$ as follows.
\[
\xy
(-8,0)*+{L}="2";
(8,0)*+{K_B}="4";
(22,0)*+{S}="6";
(-8,-12)*+{W_Y}="8";
(8,-12)*+{Y}="10";
(25,-12)*+{V_Y[1]}="12";
{\ar^{h} "2";"4"};
{\ar^{} "4";"6"};
{\ar_{} "2";"8"};
{\ar^{y} "4";"10"};
{\ar_{w_Y} "8";"10"};
{\ar_(0.46){v_Y} "10";"12"};
{\ar@{}|\circlearrowright "2";"10"};
\endxy
\]
Thus $v_Y\ci y$ should factor through some morphism $S\to V_Y[1]$. However, since $\C(S,V_Y[1])=0$, this implies $v_Y\ci y=0$. Thus $y$ factors through $w_Y$, which means $\und{y}=0$.

Exactness of $(\ref{SeqE})$ follows from {\rm (1)} and {\rm (2)}.
\end{proof}
\begin{prop}\label{PropCokThree}
Let $C[-1]\to A\ov{f}{\lra}B\ov{g}{\lra}C$ be a distinguished triangle. Suppose $C$ admits a coreflection triangle
\[
\xy
(-24,0)*+{V}="0";
(-8,0)*+{S}="2";
(8,0)*+{C}="4";
(28,0)*+{V[1]}="6";
(18,-12)*+{T}="8";
(18,-5)*+{_{\circlearrowright}}="10";
{\ar^{} "0";"2"};
{\ar^{} "2";"4"};
{\ar^{} "4";"6"};
{\ar_{} "4";"8"};
{\ar_{} "8";"6"};
\endxy
\]
satisfying $S\in\Scal$. Then, $H(f)$ is epimorphic in $\HW$.
\end{prop}
\begin{proof}
Take a coreflection triangle
\[
\xy
(-24,0)*+{V_B}="0";
(-8,0)*+{K_B}="2";
(8,0)*+{B}="4";
(28,0)*+{V_B[1]}="6";
(18,-12)*+{T_B}="8";
(18,-5)*+{_{\circlearrowright}}="10";
{\ar^{} "0";"2"};
{\ar^{k_B} "2";"4"};
{\ar^{} "4";"6"};
{\ar_{} "4";"8"};
{\ar_{} "8";"6"};
\endxy.
\]
Remark that $H(k_B)$ is isomorphic.
Then, we obtain the following commutative diagram by (the dual of) Remark \ref{RemFunctRefl}.
\[
\xy
(0,14)*+{K_B}="2";
(18,14)*+{S}="4";
(-18,0)*+{A}="10";
(0,0)*+{B}="12";
(18,0)*+{C}="14";
(15,-7)*+{}="17";
(28,-7)*+{T}="18";
(0,-14)*+{V_B[1]}="22";
(18,-14)*+{V[1]}="24";
{\ar^{h} "2";"4"};
{\ar_{k_B} "2";"12"};
{\ar^{} "4";"14"};
{\ar^{f} "10";"12"};
{\ar^{g} "12";"14"};
{\ar^{} "14";"18"};
{\ar^{} "18";"24"};
{\ar^{} "12";"22"};
{\ar^{} "14";"24"};
{\ar^{} "22";"24"};
{\ar@{}|\circlearrowright "2";"14"};
{\ar@{}|\circlearrowright "12";"24"};
{\ar@{}|\circlearrowright "17";"18"};
\endxy
\]

If we complete $h$ into a distinguished triangle
\[ S[-1]\to L\ov{\ell}{\lra}K_B\ov{h}{\lra}S, \]
then we have a commutative diagram
\[
\xy
(-22,6)*+{S[-1]}="0";
(-7,6)*+{L}="2";
(6,6)*+{K_B}="4";
(20,6)*+{S}="6";
(-22,-6)*+{C[-1]}="10";
(-7,-6)*+{A}="12";
(6,-6)*+{B}="14";
(20,-6)*+{C}="16";
{\ar^{} "0";"2"};
{\ar^{\ell} "2";"4"};
{\ar^{h} "4";"6"};
{\ar^{} "0";"10"};
{\ar^{a} "2";"12"};
{\ar^{k_B} "4";"14"};
{\ar^{} "6";"16"};
{\ar_{} "10";"12"};
{\ar_{f} "12";"14"};
{\ar_{} "14";"16"};
{\ar@{}|\circlearrowright "0";"12"};
{\ar@{}|\circlearrowright "2";"14"};
{\ar@{}|\circlearrowright "4";"16"};
\endxy.
\]
By Proposition \ref{PropCokOne}, $H(\ell)$ is epimorphic. Thus $H(f)\ci H(a)=H(k_B)\ci H(\ell)$ also becomes epimorphic. This implies $H(f)$ is epimorphic.
\end{proof}

\begin{prop}\label{RemCokThree}
For any object $C\in\C$, the following are equivalent.
\begin{enumerate}
\item $H(C)=0$ and $C\in\Scal\ast\Vcal [1]$.
\item $\tm(C)\in\Ucal/\Wcal$ and $C\in\Scal\ast\Vcal [1]$.
\item $C$ admits a coreflection triangle
\[
\xy
(-24,0)*+{V}="0";
(-8,0)*+{S}="2";
(8,0)*+{C}="4";
(28,0)*+{V[1]}="6";
(18,-12)*+{T}="8";
(18,-5)*+{_{\circlearrowright}}="10";
{\ar^{} "0";"2"};
{\ar^{} "2";"4"};
{\ar^{} "4";"6"};
{\ar_{} "4";"8"};
{\ar_{} "8";"6"};
\endxy
\]
satisfying $S\in\Scal$.
\end{enumerate}
\end{prop}
\begin{proof}
The equivalence $(1)\LR(2)$ follows from Proposition \ref{RemHTauU}. Suppose {\rm (3)} holds. This implies $\tm(C)\cong S\in\Ucal/\Wcal$, and also $C\in\Scal\ast\Vcal [1]$.

Conversely, suppose {\rm (2)} holds. By $\tm(C)\in\Ucal/\Wcal$, this $C$ admits a coreflection triangle
\[
\xy
(-24,0)*+{V}="0";
(-8,0)*+{U}="2";
(8,0)*+{C}="4";
(28,0)*+{V[1]}="6";
(18,-12)*+{T}="8";
(18,-5)*+{_{\circlearrowright}}="10";
{\ar^{} "0";"2"};
{\ar^{} "2";"4"};
{\ar^{} "4";"6"};
{\ar_{} "4";"8"};
{\ar_{} "8";"6"};
\endxy
\]
satisfying $U\in\Ucal$. By $C\in\Scal\ast\Vcal [1]$, it can be also decomposed into a distinguished triangle
\[ V_0\to S\to C\to V_0[1] \]
satisfying $S\in\Scal, V_0\in\Vcal$. Then by $\C(U,V_0[1])=0$, we obtain the following commutative diagram.
\[
\xy
(25,7)*+{V[1]}="0";
(7,7)*+{C}="2";
(-8,7)*+{U}="4";
(-23,7)*+{V}="6";
(15,6)*+{}="7";
(15,16)*+{T}="8";
(25,-6)*+{V_0[1]}="10";
(7,-6)*+{C}="12";
(-8,-6)*+{S}="14";
(-23,-6)*+{V_0}="16";
{\ar^{} "2";"0"};
{\ar^{} "4";"2"};
{\ar^{} "6";"4"};
{\ar^{} "8";"0"};
{\ar^{} "2";"8"};
{\ar^{} "0";"10"};
{\ar@{=} "2";"12"};
{\ar_{} "4";"14"};
{\ar^{} "6";"16"};
{\ar_{} "12";"10"};
{\ar_{} "14";"12"};
{\ar_{} "16";"14"};
{\ar@{}|\circlearrowright "0";"12"};
{\ar@{}|\circlearrowright "2";"14"};
{\ar@{}|\circlearrowright "4";"16"};
{\ar@{}|\circlearrowright "7";"8"};
\endxy
\]
This gives the following coreflection triangle, and thus {\rm (3)} holds.
\[
\xy
(-24,0)*+{V_0}="0";
(-8,0)*+{S}="2";
(8,0)*+{C}="4";
(28,0)*+{V_0[1]}="6";
(18,-12)*+{T}="8";
(18,-5)*+{_{\circlearrowright}}="10";
{\ar^{} "0";"2"};
{\ar^{} "2";"4"};
{\ar^{} "4";"6"};
{\ar_{} "4";"8"};
{\ar_{} "8";"6"};
\endxy
\]
\end{proof}

\begin{cor}\label{CorCok}
Let $C[-1]\ov{e}{\lra}A\ov{f}{\lra}B\ov{g}{\lra}C$ be a distinguished triangle. Suppose $C$ satisfies $H(C)=0$ and $C\in\Scal\ast\Vcal [1]$.
Then,
\[ H(C[-1])\ov{H(e)}{\lra}H(A)\ov{H(f)}{\lra}H(B)\to 0 \]
becomes a cokernel sequence in $\HW$.
\end{cor}
\begin{proof}
By Proposition \ref{PropCokThree}, the morphism $H(f)$ is epimorphic.

Take a coreflection triangle
\[
\xy
(-24,0)*+{V}="0";
(-8,0)*+{S}="2";
(8,0)*+{C}="4";
(28,0)*+{V[1]}="6";
(18,-12)*+{T}="8";
(18,-5)*+{_{\circlearrowright}}="10";
{\ar^{} "0";"2"};
{\ar^{k} "2";"4"};
{\ar^{} "4";"6"};
{\ar_{} "4";"8"};
{\ar_{} "8";"6"};
\endxy
\]
satisfying $S\in\Scal$. If we complete $s=-e\ci k[-1]$ into a distinguished triangle
\[ S[-1]\ov{s}{\lra}A\ov{d}{\lra}D\to S, \]
then by the octahedron axiom, we obtain a commutative diagram
\[
\xy
(-16,16)*+{S[-1]}="0";
(-15,-1)*+{C[-1]}="2";
(-14,-19)*+{V}="4";
(-4,2)*+{A}="6";
(2.3,-5.3)*+{D}="8";
(17,8.5)*+{B}="10";
(-8.5,-7.5)*+_{_{\circlearrowright}}="12";
(-11.5,5.5)*+_{_{\circlearrowright}}="14";
(4.3,0.5)*+_{_{\circlearrowright}}="14";
{\ar_{-k[-1]} "0";"2"};
{\ar_{} "2";"4"};
{\ar^{s} "0";"6"};
{\ar_(0.56){e} "2";"6"};
{\ar_{d} "6";"8"};
{\ar^{f} "6";"10"};
{\ar_{} "4";"8"};
{\ar_{h} "8";"10"};
\endxy
\]
in which
\begin{eqnarray*}
&S[-1]\to A\to D\to S,\quad S[-1]\to C[-1]\to V\to S,&\\
&C[-1]\to A\to B\to C,\quad V\to D\to B\to V[1]&
\end{eqnarray*}
are distinguished triangles.

Applying Proposition \ref{PropCokOne} and its dual to the triangles
\begin{eqnarray*}
&S[-1]\ov{s}{\lra}A\ov{d}{\lra}D\to S,&\\
&S[-1]\ov{-k[-1]}{\lra}C[-1]\to V\to S,&\\
&V\to D\ov{h}{\lra}B\to V[1],&
\end{eqnarray*}
we obtain exact sequences
\begin{eqnarray*}
&H(S[-1])\ov{H(s)}{\lra}H(A)\ov{H(d)}{\lra}H(D)\to 0\quad \exact,&\\
&H(S[-1])\ov{-H(k[-1])}{\lra}H(C[-1])\to 0\quad\exact,&\\
&0\to H(D)\ov{H(h)}{\lra}H(B)\quad\exact.&
\end{eqnarray*}
These form a commutative diagram
\[
\xy
(-24,14)*+{H(S[-1])}="0";
(-24,0)*+{H(C[-1])}="2";
(-24,-12)*+{0}="4";
(0,-13)*+{0}="5";
(0,0)*+{H(A)}="6";
(12.5,-8)*+{H(D)}="8";
(29,0)*+{H(B)}="10";
(23.5,-14)*+{0}="12";
{\ar_{-H(k[-1])} "0";"2"};
{\ar_{} "2";"4"};
{\ar^{H(s)} "0";"6"};
{\ar_(0.56){H(e)} "2";"6"};
{\ar_{H(d)} "6";"8"};
{\ar^{H(f)} "6";"10"};
{\ar_{} "5";"8"};
{\ar_{H(h)} "8";"10"};
{\ar_{} "8";"12"};
%
(-16,4)*+_{_{\circlearrowright}}="24";
(13,-3)*+_{_{\circlearrowright}}="26";
\endxy,
\]
which implies the exactness of
\[ H(C[-1])\ov{H(e)}{\lra}H(A)\ov{H(f)}{\lra}H(B). \]
Together with the epimorphicity of $H(f)$, we obtain the conclusion.
\end{proof}

Corollary \ref{CorCok} gives a sufficient condition for Problem \ref{ProbTwo}.
\begin{cor}\label{CorProT}
Let $\Pcal=\STUV$ and $\Pcal\ppr=\STUVp$ be twin cotorsion pairs on $\C$. If they satisfy the conditions
\[ H\ppr(\Ucal)=H\ppr(\Tcal)=0 \]
and
\[ \Scal\se\Scal\ppr\ast\Vcal\ppr [1],\quad \Vcal\se\Scal\ppr [-1]\ast\Vcal\ppr, \]
then $E=\Hbarp\ci\iota\co\HW\to\HWp$ satisfies $E\ci H\cong H\ppr$.
\end{cor}
\begin{proof}
For each object $A\in\C$, choose a reflection triangle
\begin{equation}\label{ReflTriALas}
\xy
(-28,0)*+{S_A[-1]}="0";
(-8,0)*+{A}="2";
(8,0)*+{Z_A}="4";
(24,0)*+{S_A}="6";
(-18,-12)*+{U_A}="8";
(-18,-5)*+{_{\circlearrowright}}="10";
{\ar^{s_A} "0";"2"};
{\ar^{z_A} "2";"4"};
{\ar_{} "4";"6"};
{\ar_{} "0";"8"};
{\ar_{} "8";"2"};
\endxy
\end{equation}
and a coreflection triangle
\[
\xy
(-24,0)*+{V_A}="0";
(-8,0)*+{K_A}="2";
(8,0)*+{A}="4";
(28,0)*+{V_A[1]}="6";
(18,-12)*+{T_A}="8";
(18,-5)*+{_{\circlearrowright}}="10";
{\ar^{} "0";"2"};
{\ar^{k_A} "2";"4"};
{\ar^{} "4";"6"};
{\ar_{} "4";"8"};
{\ar_{} "8";"6"};
\endxy.
\]
Similarly as in the proof of Proposition \ref{PropSingle}, it suffices to show that $H\ppr(z_A)$ and $H\ppr(k_A)$ become isomorphisms in $\HWp$.

By assumption, $S_A$ satisfies $H\ppr(S_A)=0$ and $S_A\in\Scal\ppr\ast\Vcal\ppr [1]$. Thus, applying Corollary \ref{CorCok} to $(\ref{ReflTriALas})$, we obtain a cokernel sequence
\[ H\ppr(S_A[-1])\ov{H\ppr(s_A)}{\lra}H\ppr(A)\ov{H\ppr(z_A)}{\lra}H\ppr(Z_A)\to 0. \]
Since $H\ppr(s_A)$ factors through $H\ppr(U_A)=0$, this means $H\ppr(z_A)$ is an isomorphism.

Dually, we have a kernel sequence
\[ 0\to H\ppr(K_A)\ov{H\ppr(k_A)}{\lra}H\ppr(A)\ov{0}{\lra}H\ppr(V_A[1]), \]
which means $H\ppr(k_A)$ is an isomorphism.
\end{proof}

If we rephrase this in the setting of Problem \ref{ProbOne}, we obtain the following.
\begin{thm}\label{MainThm}
Let $\Pcal=\STUV$ and $\Pcal\ppr=\STUVp$ be twin cotorsion pairs on $\C$ and $\C\ppr$. Let $F\co\C\ov{\simeq}{\lra}\C\ppr$ be a triangle equivalence. If $\Pcal$ and $\Pcal\ppr$ satisfy the conditions
\begin{itemize}
\item[{\rm (A)}] \ $H\ppr F(\Ucal)=H\ppr F(\Tcal)=0$, \ $HF\iv(\Ucal\ppr)=HF\iv(\Tcal\ppr)=0$,
\item[{\rm (B)}] \ $F(\Scal)\se\Scal\ppr\ast\Vcal\ppr [1]$, \ $F(\Vcal)\se\Scal\ppr [-1]\ast\Vcal\ppr$,
\item[{\rm (C)}] \ $F\iv(\Scal\ppr)\se\Scal\ast\Vcal [1]$, \ $F\iv(\Vcal\ppr)\se\Scal [-1]\ast\Vcal$,
\end{itemize}
then $\Pcal$ is heart-equivalent to $\Pcal\ppr$ along $F$.
\end{thm}
\begin{proof}
This follows immediately from Corollary \ref{CorProT}.
\end{proof}

\section{An application}

In the rest, we assume the following.
\begin{assumption}
$\ $
\begin{enumerate}
\item $\C$ is $k$-linear for some field $k$, and has a Serre functor $\Sbf$.
\item $\Dcal\se\C$ is a functorially finite rigid subcategory, closed under isomorphisms, finite direct sums and summands.
\item $\UVp$ is a cotorsion pair satisfying $\Dcal\se\Ucal\ppr$.
\item $\Ucal=\mu\iv(\Ucal\ppr ;\Dcal)=(\Dcal\ast\Ucal\ppr [1])\cap {}\ppp\Dcal [1]$.
\item $\Ucal\ppr=\mu(\Ucal ;\Dcal)=(\Ucal [-1]\ast\Dcal)\cap \Dcal [-1]\ppp$.
\item $({}\ppp\Ucal [1],\Ucal)$ is a cotorsion pair.
\end{enumerate}
And, additionally,
\begin{enumerate}
\setcounter{enumi}{6}
\item $\Ucal\se\Dcal [-1]\ppp$.
\item $\Ucal\ppr\se {}\ppp\Dcal [1]$.
\end{enumerate}
Remark that {\rm (7)} and {\rm (8)} are automatically satisfied if $\C$ is 2-Calabi-Yau.
\end{assumption}

\begin{prop}\label{PropPropLast}
Put $F=\Sbf\ci [-2]\co\C\ov{\simeq}{\lra}\C$. With the above assumption, the twin cotorsion pair $(({}\ppp\Ucal [1],\Ucal),({}\ppp\Dcal [1],\Dcal))$ is heart-equivalent to $((\Dcal,\Dcal [-1]\ppp),(\Ucal\ppr,\Vcal\ppr))$ along $F$.
\end{prop}
This recovers the following result in \cite{MP}.
\begin{fact}
(Theorem 2.9 in \cite{MP})

Assume $\C$ is a Krull-Schmidt $k$-linear $\Hom$-finite triangulated category with a Serre functor. Let $T$ be a basic rigid object, and let
\[ T=\ovl{T}\oplus R,\qquad T\ppr=\ovl{T}\oplus R^{\ast} \]
be direct sums in $\C$, where $R$ and $R^{\ast}$ are related by a distinguished triangle
\[ R^{\ast}\to B\to R\to R^{\ast}[1] \]
in which $B\in\add\ovl{T}$, and $B\to R$ is a minimal right $\add\ovl{T}$-approximation. Then there is an equivalence
\[ \big((\add T)\ast(\add\ovl{T}[1])\big)/\add T\, \simeq\,  \big((\add \ovl{T}[-1])\ast(\add T\ppr)\big)/\add T\ppr. \]

Indeed, this equivalence follows from Proposition \ref{PropPropLast}, if we put $\Dcal=\add\overline{T},\, \Ucal=\add T$ and $\Ucal\ppr=\add T\ppr$.
\end{fact}


\begin{proof} ({\it of Proposition \ref{PropPropLast}.})
We confirm the conditions {\rm (A),(B),(C)} in Theorem \ref{MainThm}.
Obviously, we have
\begin{eqnarray*}
&F(\Ucal)\se F({}\ppp\Dcal [1])=\Sbf({}\ppp\Dcal [-1])=\Dcal [-1]\ppp,&\\
&F\iv(\Ucal\ppr)\se F\iv(\Dcal [-1]\ppp)=\Sbf\iv(\Dcal [1]\ppp)={}\ppp\Dcal [1],&
\end{eqnarray*}
and thus {\rm (A)} is satisfied.
It remains to show the following.
\begin{itemize}
\item[{\rm (B1)}] \ $F({}\ppp\Ucal [1])\se\Dcal\ast\Vcal\ppr [1]$.
\item[{\rm (B2)}] \ $F(\Dcal)\se\Dcal [-1]\ast\Vcal\ppr$.
\item[{\rm (C1)}] \ $F\iv(\Dcal)\se {}\ppp\Ucal [1]\ast\Dcal [1]$.
\item[{\rm (C2)}] \ $F\iv(\Vcal\ppr)\se {}\ppp\Ucal\ast\Dcal$.
\end{itemize}

\medskip

\noindent {\bf [Confirmation of {\rm (B1)}]} This requires condition {\rm (8)}. Since $F({}\ppp\Ucal [1])=\Ucal [-1]\ppp$, it is equivalent to show $\Ucal [-1]\ppp\se\Dcal\ast\Vcal\ppr [1]$.

For any $X\in\Ucal [-1]\ppp$, decompose it into a distinguished triangle
\[ V_X\ppr\to U_X\ppr\to X\to V_X\ppr [1]\quad(U_X\ppr\in\Ucal\ppr,V_X\ppr\in\Vcal\ppr) \]
and then, $U_X\ppr$ into
\[ U[-1]\to U_X\ppr\to D\to U\quad(U\in\Ucal,D\in\Dcal). \]
By $X\in\Ucal [-1]\ppp$, we obtain a diagram
\[
\xy
(-16.2,16)*+{U_X\ppr}="0";
(-14.8,-1)*+{D}="2";
(-13.5,-17.5)*+{U}="4";
(-4,2)*+{X}="6";
(2,-5.5)*+{V_X\ppr [1]}="8";
(16.7,8.5)*+{P}="10";
(-8.5,-7.5)*+_{_{\circlearrowright}}="12";
(-11.5,5.5)*+_{_{\circlearrowright}}="14";
(4.3,0.5)*+_{_{\circlearrowright}}="14";
{\ar_{} "0";"2"};
{\ar_{} "2";"4"};
{\ar^{} "0";"6"};
{\ar_{} "2";"6"};
{\ar_{} "6";"8"};
{\ar^{} "6";"10"};
{\ar_{} "4";"8"};
{\ar_{} "8";"10"};
\endxy
\]
in which
\begin{eqnarray*}
&U\ppr_X\to D\to U\to U\ppr_X[1],\quad U\ppr_X\to X\to V\ppr_X[1]\to U\ppr_X[1],&\\
&U\to V\ppr_X[1]\to P\to U[1],\quad D\to X\to P\to D[1]&
\end{eqnarray*}
are distinguished triangles.
Then for any $U\ppr\in\Ucal\ppr$ and any $u\ppr\in\C(U\ppr,P)$, the diagram
\[
\xy
(8,14)*+{U\ppr}="0";
(-8,0)*+{V_X\ppr [1]}="2";
(8,0)*+{P}="4";
(28,0)*+{U[1]}="6";
(18,-12)*+{D[1]}="8";
(18,-5)*+{_{\circlearrowright}}="10";
{\ar^{u\ppr} "0";"4"};
{\ar^{} "2";"4"};
{\ar^{} "4";"6"};
{\ar_{} "4";"8"};
{\ar_{} "8";"6"};
\endxy
\]
and the condition $\Ucal\ppr\se {}\ppp\Dcal [1]$ together with $\C(U\ppr,V\ppr_X[1])=0$ show $u\ppr=0$. Namely $P$ belongs to $\Ucal^{\prime\perp}=\Vcal\ppr [1]$, and thus it follows $X\in\Dcal\ast\Vcal\ppr [1]$.

\medskip

\noindent {\bf [Confirmation of {\rm (B2)}]}
We have
\begin{eqnarray*}
\C(\Ucal\ppr [-1],F(\Dcal))&=&\C(\Ucal\ppr [-1],\Sbf(\Dcal [-2]))%
\ \cong\ \C(\Dcal [-2],\Ucal\ppr [-1])^{\vee}\\
&\cong&\C(\Dcal [-1],\Ucal\ppr)^{\vee}\ =\ 0.
\end{eqnarray*}
This means $F(\Dcal)\se\Vcal\ppr$. Here, ${(-)}^{\vee}$ denotes the $k$-dual.

\medskip

\noindent {\bf [Confirmation of {\rm (C1)}]}
We have
\begin{eqnarray*}
\C(F\iv(\Dcal),\Ucal [1])&\cong&\C(\Dcal [2],\Sbf(\Ucal [1]))
\ \cong\ \C(\Ucal [1],\Dcal [2])^{\vee}\\
&\cong&\C(\Ucal,\Dcal [1])^{\vee}\ =\ 0.
\end{eqnarray*}
This means $F\iv(\Dcal)\se {}\ppp\Ucal [1]$.

\medskip

\noindent {\bf [Confirmation of {\rm (C2)}]}
This requires condition {\rm (7)}. Since $F\iv(\Vcal\ppr)=F\iv(\Ucal\ppr [-1]\ppp)={}\ppp\Ucal\ppr [1]$, it is equivalent to show ${}\ppp\Ucal\ppr [1]\se {}\ppp\Ucal\ast\Dcal$.

For any $X\in\ppp\Ucal\ppr [1]$, decompose it into a distinguished triangle
\[ Q_X\to X\to U_X\to Q_X[1]\quad(Q_X\in {}\ppp\Ucal,U_X\in\Ucal) \]
and then, $U_X$ into
\[ D\to U_X\to U\ppr [1]\to D[1]\quad(D\in \Dcal,U\ppr\in\Ucal\ppr). \]
By $X\in {}\ppp\Ucal\ppr [1]$, we obtain a diagram
\[
\xy
(-20,16)*+{R}="0";
(0.5,15)*+{Q_X}="2";
(17,14)*+{U\ppr}="4";
(-2,4)*+{X}="6";
(6.1,-0.6)*+{D}="8";
(-5.8,-14.2)*+{U_X}="10";
(7.5,8.5)*+_{_{\circlearrowright}}="12";
(-5.5,11.5)*+_{_{\circlearrowright}}="14";
(-0.5,-4.3)*+_{_{\circlearrowright}}="14";
{\ar^{} "0";"2"};
{\ar^{} "2";"4"};
{\ar_{} "0";"6"};
{\ar^{} "2";"6"};
{\ar^{} "6";"8"};
{\ar_{} "6";"10"};
{\ar^{} "4";"8"};
{\ar^{} "8";"10"};
\endxy
\]
in which\begin{eqnarray*}
&R\to Q_X\to U\ppr\to R[1],\quad R\to X\to D\to R[1],&\\
&Q_X\to X\to U_X\to Q_X[1],\quad U\ppr\to D\to U_X\to U\ppr [1]&
\end{eqnarray*}
are distinguished triangles.
Then for any $U\in\Ucal$ and any $u\in\C(R,U)$ the diagram
\[
\xy
(-18,12)*+{D[-1]}="0";
(-28,0)*+{U\ppr [-1]}="2";
(-8,0)*+{R}="4";
(8,0)*+{Q_X}="6";
(-8,-14)*+{U}="8";
(-18,5)*+{_{\circlearrowright}}="10";
{\ar_{} "0";"4"};
{\ar_{} "2";"0"};
{\ar^{} "2";"4"};
{\ar^{} "4";"6"};
{\ar^{u} "4";"8"};
\endxy
\]
and the condition $\Ucal\se\Dcal [-1]\ppp$ shows $u=0$. Namely $R$ belongs to ${}\ppp\Ucal$, and thus we have $X\in {}\ppp\Ucal\ast\Dcal$.
\end{proof}

\section*{Acknowledgement}
This article has been written when the author was staying at LAMFA, l'Universit\'{e} de Picardie-Jules Verne, by the support of JSPS Postdoctoral Fellowships for Research Abroad. He wishes to thank the hospitality of Professor Serge Bouc, Professor Radu Stancu and the members of LAMFA.


\begin{thebibliography}{BBD}                                                    \bibitem[AN]{AN} Abe, N.; Nakaoka, H.: \emph{General heart construction on a triangulated category (II): associated homological functor}, Appl. Categ. Structures, \textbf{20} (2012) no.2, 161--174.

\bibitem[BBD]{BBD}Be\u{\i}linson, A. A.; Bernstein, J.; Deligne, P.: \emph{Faisceaux pervers} (French) [Perverse sheaves] Analysis and topology on singular spaces, I (Luminy, 1981), 5--171, Ast\'{e}risque, \textbf{100}, Soc. Math. France, Paris, 1982.  



\bibitem[BM]{BM}Buan, A. B.; Marsh, R. J.: \emph{From triangulated categories to module categories via localization II: calculus of fractions}, J. Lond. Math. Soc. (2) \textbf{86} (2012) no. 1, 152--170

\bibitem[BMR]{BMR}Buan, A. B.; Marsh, R. J.; Reiten, I.: \emph{Cluster-tilted algebras}, Trans. Amer. Math. Soc. \textbf{359} (2007) no. 1, 323--332.

\bibitem[IY]{IY}Iyama, O.; Yoshino, Y.: \emph{Mutation in triangulated categories and rigid Cohen-Macaulay modules}, Invent. Math. \textbf{172} (2008) no. 1, 117--168. 

\bibitem[KR]{KR} Keller, B.; Reiten, I.: \emph{Cluster-tilted algebras are Gorenstein and stably Calabi-Yau}, Adv. Math. \textbf{211} (2007) no. 1, 123--151.

\bibitem[KZ]{KZ}Koenig, S.; Zhu, B.: \emph{From triangulated categories to abelian categories: cluster tilting in a general framework}, Math. Z. \textbf{258} (2008) no. 1, 143--160.

\bibitem[L1]{L1} Liu, Y.: \emph{Hearts of twin cotorsion pairs on exact categories}, J. Algebra, \textbf{394} (2013) 245--284.

\bibitem[L2]{L2} Liu, Y.: \emph{Half exact functors associated with general hearts on exact categories}, arxiv: 1305.1433.

\bibitem[MP]{MP} Marsh, R.; Palu, Y.: \emph{Nearly Morita equivalences and rigid objects}, arxiv: 1405.7061.

\bibitem[N1]{N1} Nakaoka, H.: \emph{General heart construction on a triangulated category (I): unifying $t$-structures and cluster tilting subcategories}, Appl. Categ. Structures, \textbf{19} (2011) no.6, 879--899.

\bibitem[N2]{N2} Nakaoka, H.: \emph{General heart construction for twin torsion pairs on triangulated categories}, J. Algebra, \textbf{374} (2013) 195--215.



\bibitem[ZZ]{ZZ}Zhou, Y.; Zhu, B.: \emph{$T$-structures and torsion pairs in a 2-Calabi-Yau triangulated category},  J. Lond. Math. Soc. (2) 89 (2014), no. 1, 213--234.


\end{thebibliography}
\end{document}